\documentclass[reqno]{amsart}
\usepackage{amssymb, amsmath, amsthm, amscd, mathrsfs}

\usepackage{paralist}
\usepackage{cite}
\usepackage{bigints}
\usepackage{dsfont}
\usepackage{empheq}
\usepackage{amssymb}
\usepackage{cases}
\usepackage{enumitem}
\allowdisplaybreaks

\usepackage{caption}
\usepackage{booktabs}

\usepackage{float}
\usepackage[ruled,vlined,linesnumbered,noresetcount]{algorithm2e}

\usepackage[colorlinks=true]{hyperref}
\numberwithin{equation}{section}

\theoremstyle{plain}
\newtheorem{theorem}{Theorem}[section]

\newtheorem{proposition}[theorem]{Proposition}
\theoremstyle{definition}
\newtheorem{definition}[theorem]{Definition}
\newtheorem{remark}{Remark}

\newcommand{\vertiii}[1]{{\left\vert\kern-0.25ex\left\vert\kern-0.25ex\left\vert #1 
    \right\vert\kern-0.25ex\right\vert\kern-0.25ex\right\vert}}

\title[Uniqueness for the singular Schrödinger equation] %Use the shortened version of the full title
      {Uniqueness for the Schrödinger equation with an inverse square potential and application to controllability and inverse problems}

\author{Salah-Eddine Chorfi}

\address{S. E. Chorfi, Faculty of Sciences Semlalia, Cadi Ayyad University, LMDP, IRD, UMMISCO, B.P. 2390, Marrakesh, Morocco}
\email{s.chorfi@uca.ac.ma}

\makeatletter %added for 2020 MSC
\@namedef{subjclassname@2020}{%
  \textup{2020} Mathematics Subject Classification}
\makeatother

\subjclass[2020]{Primary: 93B05, 35R30; Secondary: 35J10, 35A21, 35J75}
\keywords{Schrödinger equation, singular potential, inverse square potential, uniqueness, approximate controllability}

\thanks{This work started when the author visited the University of Tokyo as an invited professor. The author thanks Masahiro Yamamoto for the hospitality, support, and fruitful discussions.}
 
\begin{document}
\begin{abstract}
In this paper, we prove a sharp uniqueness result for the singular Schrödinger equation with an inverse square potential. This will be done without assuming geometrical restrictions on the observation region. The proof relies on a recent technique transforming the Schrödinger equation into an elliptic equation. We show that this technique is still applicable for singular equations. In our case, substantial difficulties arise when dealing with singular potentials of cylindrical type. Using the uniqueness result, we show the approximate controllability of the equation using a distributed control. The uniqueness result is also applied to prove the uniqueness for an inverse source problem.
\end{abstract}
\dedicatory{\large Dedicated to the memory of Professor Hammadi Bouslous}
\maketitle

\section{Introduction and main results}
The Schrödinger equation is a prototype model in quantum mechanics that describes the evolution of a quantum mechanical system over time. Singular potentials usually represent some potential energy functions that become unbounded at a single or several positions. A common example that appears in many applications is given by the inverse square potential. This is the case, for instance, when the quantum mechanical system under consideration is governed by a potential energy field that decreases as the inverse square of the distance from the origin point. Mathematically, this type of singular potentials at the origin is given by $V(x)=\dfrac{\lambda}{r^2}$, where $\lambda$ is a constant and $r=|x|$ is the distance from the origin.

To present the problem, let $n \ge 1$ be an integer and let $T>0$ be a positive terminal time. We assume that $\Omega \subset \mathbb R^n$ is a bounded domain containing the origin, i.e., $0 \in \Omega$, and its boundary $\Gamma:=\partial \Omega$ is of class $C^2$. Let $\nu(x)$ be the unit normal vector at $x\in \Gamma$ pointing towards the exterior of $\Omega$. We consider the following singular Schrödinger equation with the so-called inverse square potential
\begin{equation}\label{(1.2)}
    \mathrm{i}\partial_t u -\Delta u-\frac{\lambda}{|x|^2} u =0, \qquad (t, x) \in Q,
\end{equation}
where $\mathrm{i}:= \sqrt{-1}$, $Q:=(0, T) \times \Omega$, $\Sigma:=(0, T) \times \partial \Omega$, and $\lambda < \lambda_\star$ (subcritical case), with $\lambda_\star(n):=\dfrac{(n-2)^2}{4}$ is the critical constant in the Hardy inequality when $n \neq 2$: for every $y \in H_0^1(\Omega)$, we have $\dfrac{y}{|x|} \in L^2(\Omega)$ and
\begin{equation}\label{hp}
    \forall y \in H_0^1(\Omega), \quad \lambda_{\star} \int_{\Omega} \frac{y^2}{|x|^2} \, d x \leq \int_{\Omega}|\nabla y|^2\, d x .
\end{equation}

Control and inverse problems for Schrödinger equations with bounded coefficients have been extensively studied in the literature. Let us mention some relevant few works: Lebeau \cite{Le'92} proved that the so-called Geometric Control Condition (GCC) is sufficient for the exact controllability of the standard Schrödinger equation at any time $T$. Machtyngier \cite{Ma'94} proved an exact controllability result for the Schrödinger equation by the multiplier method. We refer to \cite{Zu'03} for a short review on this topic. Martin et al. \cite{MRR} have used flatness to construct a smooth control for the one-dimensional Schrödinger equation. More recently, Mercado and Morales \cite{MM'23} have proven the exact controllability of a Schrödinger equation with dynamic boundary conditions. As for inverse problems by Carleman estimates, Baudouin and Puel \cite{BP'08, BP'02} have proven a stability estimate for bounded potentials. Baudouin and Mercado \cite{BM'08} have considered a discontinuous main coefficient. Moreover, in \cite{MOR'08}, Mercado et al. have obtained some stability results under less restrictive measurements. Huang et al. \cite{HKYM'19} have established Lipschitz stability of the electric potential and the direction of the magnetic field. More recently, Imanuvilov and Yamamoto \cite{IY'022} have proven a Lipschitz stability estimate for source terms or zeroth order coefficients in Schrödinger equation from data taken at the terminal time $T$.

In the last decades, an increasing interest has been devoted to the study of Schrödinger equations with unbounded coefficients and singular potentials. We mention the paper by Yuan and Yamamoto \cite{YM} on the uniqueness for less regular potentials. Baudouin et al. \cite{BKP'05} have studied the regularity of a Schrödinger equation involving a time-dependent Coulomb potential with application to a bilinear optimal control problem. Special attention was paid to inverse square potentials, also known as Hardy-type potentials. However, only a few works have been devoted to this critical type of potentials, for instance, Vancostenoble and Zuazua \cite{VZ'09} have proven the exact boundary controllability for an interior singularity. In \cite{Ca'12}, Cazacu considered the same controllability problem for a boundary singularity. It should be emphasized that almost all previous works assume some geometric restrictions on the observation region. For instance, the authors in \cite{VZ'09} considered $\Omega$ to be a bounded domain with boundary $\Gamma$ of class $C^3$ such that $0\in \Omega$. Moreover, the observation region $\omega \subset \Omega$ is chosen sufficiently large so that $\omega$ is a neighborhood of $\overline{\Gamma}_0$, where
$$\Gamma_0:=\left\{x\in \Gamma \colon x\cdot \nu(x) \ge 0\right\},$$
that is $\omega =\Omega \cap U$, where $U\subset \mathbb R^n$ is an open set such that $\overline{\Gamma}_0 \subset U$. Note that if $\Omega$ is convex, we have $\Gamma_0 =\Gamma$, and then one should observe through the whole boundary. We also refer to \cite{XWD23} for recent results on the half-line. As for the parabolic operator, we refer to the recent work \cite{SV23} for some results employing a local Carleman estimate near the singularity. Furthermore, for controllability of degenerate and singular parabolic equations, we refer to the book by Fragnelli and Mugnai \cite{FM'21}. Fragnelli et al. \cite{FMS'24} have recently investigated controllability and stabilization of a degenerate/singular Schrödinger equation using multiplier and compactness-uniqueness arguments.

The key idea that we adopt in this paper relies on transforming the singular Schrödinger equation into a singular elliptic equation. Regarding elliptic equations with unbounded coefficients, unique continuation results have been established in the literature. For instance, Jerison and Kenig \cite{JK'85} proved (strong) unique continuation for the differential inequality
\begin{equation}\label{diffi}
    |\Delta w(x)| \le |V(x)| |w(x)|, \qquad x\in \Omega,
\end{equation}
for potentials $V\in L^{\frac{n}{2}}(\Omega)$ which is optimal for $L^p$ spaces. We also refer to the recent work by Choulli \cite{Ch23} for a quantitative version. However, unique continuation holds for the inverse square potential $V(x)=\dfrac{\lambda}{|x|^2}$, see Pan \cite{Pa'92}. Note that $\dfrac{\lambda}{|x|^2} \notin L^{\frac{n}{2}}(\Omega)$, which makes the inverse square potential regarded as critical and rather challenging compared to other types of singularity. The strategy we adopt here has been recently developed in \cite{IY'22} for a Schrödinger equation with bounded first and zeroth-order coefficients, that is, for potentials $V(x)\in L^\infty(\Omega)$. In our case, after transformation, we obtain a singular elliptic equation at the expense of augmenting the variable by one dimension. This gives rise to cylindrical singular potentials for which the aforementioned results do not apply.

The first main result consists of sharp uniqueness for the singular Schrödinger equation.
\begin{theorem}\label{thm1}
Let $\omega \subset \Omega$ be an arbitrary nonempty open set. We assume that $u \in C([0,T];L^2(\Omega))$ satisfies \eqref{(1.2)}.
Then $u = 0$ in $(0,T) \times \omega$ yields $u=0$ in $Q$.
\end{theorem}

The structure of the paper will be as follows: in Section \ref{sec2}, we discuss some unique continuation results for an elliptic equation with a cylindrical singular potential. Using some recent results, we prove a unique continuation principle. Section \ref{sec3} is devoted to the proof of Theorem \ref{thm1}. The result will then be applied to establish the approximate controllability of the singular Schrödinger equation and the uniqueness in an inverse source problem.

\section{Uniqueness for an elliptic equation with a singular potential}\label{sec2}
In this section, we discuss some results (which will be used later) for the unique continuation of an elliptic equation with a singular cylindrical potential. We refer to \cite{FFT'12} for a much more general setting.

\subsection*{Notations}
Let us first recall some notations that will be used throughout the paper. We set $D=(D_0,D'), D'=(D_1,\dots, D_n), D_0=-\mathrm{i}\partial_t, D_j=-\mathrm{i}\partial_{x_j}$ for $j\in \{1,\dots, n\}$. We will use the following functions space:
$$H^{1,2}(Q):=H^1\left(0, T ; L^2(\Omega)\right) \cap L^2\left(0, T; H^2(\Omega)\right).$$
If $P$ is a differential operator, we consider its formal adjoint $P^*$:
$$\left \langle P y, z \right \rangle_{L^2(E)}=\left \langle y, P^* z\right \rangle_{L^2(E)}, \qquad y, z \in C_0^{\infty}(E),$$
where $E \subset Q$ is a domain with a smooth boundary.

For $(t,x)\in (-1,1)\times \Omega$, we denote $\mathbf{x}:=(t,x) \in \mathbb{R}^N$, where $N=n+1$. We denote $0:=0_{\mathbb{R}^n}$, $\mathbf{0}:=0_{\mathbb{R}^N}$, and we set $2^*=\frac{2N}{N-2}>2$ for the critical Sobolev exponent. For any $r>0$, we denote by $B_r$ the open ball in $\mathbb{R}^N$ of center $\mathbf{0}$ and radius $r$, that is to say $B_r:=\left\{\mathbf{x} \in \mathbb{R}^N: |\mathbf{x}|<r\right\}$. More generally, we set $B(\mathbf{x}, r)=\left\{\mathbf{y} \in \mathbb{R}^N:|\mathbf{y}-\mathbf{x}|<r\right\}$. The closed ball will be denoted by $\overline{B}_r=\left\{\mathbf{x} \in \mathbb{R}^N:|\mathbf{x}| \leq r\right\}$, and the $(N-1)$-dimensional unit sphere by $\mathbb{S}^{N-1}:=\{\theta \in \mathbb{R}^N: |\theta|=1 \}$. The notation $e_1$ stands for the first vector in the canonical basis of $\mathbb{R}^N$. Finally, $|E|$ will denote the Lebesgue measure of a set $E \subset (-1,1) \times \Omega$.

We shall consider singular potentials of cylindrical form
$$V(\mathbf{x})=\frac{\lambda}{|x|^2}, \qquad \mathbf{x}=(t,x),$$
which can be written as
$$
V(\mathbf{x})=\frac{a\left(\frac{\mathbf{x}}{|\mathbf{x}|}\right)}{|\mathbf{x}|^2},
$$
where 
\begin{equation}
a(\theta):=\dfrac{\lambda}{|\theta_n|^2} \qquad \text{for all }\theta=\frac{\mathbf{x}}{|\mathbf{x}|} \in \mathbb{S}^{N-1} \setminus \{e_1, -e_1\}, 
\end{equation}
with $\theta_n:=\dfrac{x}{|\mathbf{x}|}$. The discrete spectrum of the operator $-\Delta_{\mathbb{S}^{N-1}}-a$ consists of a nondecreasing sequence of eigenvalues
$$
\mu_1(a) \leq \mu_2(a) \leq \cdots \leq \mu_k(a) \leq \cdots,
$$
diverging to $+\infty$.

We deal with the following elliptic equation with cylindrical singular potential
\begin{equation}\label{eqels}
    -\Delta_\mathbf{x} w(\mathbf{x}) -\frac{\lambda}{|x|^2} w(\mathbf{x})=0, \qquad \mathbf{x}:=(t,x) \in (-1,1)\times \Omega.
\end{equation}
We need the following key proposition, which is a special form of \cite[Theorem 1.1]{FFT'12}.
\begin{proposition}\label{lmasy}
Let $w \neq 0$ be an $H^1((-1,1)\times \Omega)$ solution to \eqref{eqels}. Then, there exists $k_0 \in \mathbb{N}$ such that if $m \geq 1$ is the multiplicity of the eigenvalue $\mu_{k_0}(a)$ and $\left\{\psi_i: j_0 \leq i \leq j_0+m-1\right\}$, $\left(j_0 \leq k_0 \leq j_0+m-1\right)$, is an $L^2\left(\mathbb{S}^{N-1}\right)$-orthonormal basis for the eigenspace associated to $\mu_{k_0}(a)$, then
\begin{equation}
    r^{-\gamma} w(r \mathbf{x}) \rightarrow|\mathbf{x}|^\gamma \psi\left(\frac{\mathbf{x}}{|\mathbf{x}|}\right) \quad \text { in } H^1\left(B_1\right) \quad \text { as } r \rightarrow 0^{+},
\end{equation}
where
$$
\gamma=-\frac{N-2}{2}+\sqrt{\left(\frac{N-2}{2}\right)^2+\mu_{k_0}(a)}, \quad \beta_i=R^{-\gamma} \int_{\mathbb{S}^{N-1}} w(R \theta) \psi_i(\theta) d S
$$
for all $R>0$ such that $\overline{B}_R=\left\{\mathbf{x} \in \mathbb{R}^N: |\mathbf{x}| \leq R\right\} \subset (-1,1)\times \Omega$, and
\begin{equation}\label{betai}
\psi(\theta):=\sum_{i=j_0}^{j_0+m-1} \beta_i \psi_i(\theta),\; \theta\in \mathbb{S}^{N-1}, \; \left(\beta_{j_0}, \beta_{j_0+1}, \ldots, \beta_{j_0+m-1}\right) \neq(0,0, \ldots, 0).
\end{equation}
\end{proposition}
Next, we state a unique continuation result.
\begin{proposition}\label{propuc}
Let $\omega \subset \Omega$ be a nonempty open set. Let $w$ be an $H^1((-1,1)\times \Omega)$ solution to \eqref{eqels}. If $w = 0$ in $(-1,1)\times \omega$, then $w = 0$ in $(-1,1)\times \Omega$.
\end{proposition}

\begin{remark}
The proof of Proposition \ref{propuc} can be deduced from classical results. If $0\in \omega$, it suffices to localize the singularity $0$ in a small open ball $B_R$ of radius $R>0$ so that $\overline{B}_R \subset \omega$. Then apply the unique continuation for elliptic operators with bounded potentials in $(-1,1)\times (\Omega\setminus\overline{B}_R)$; see, e.g. \cite[Theorem 5.2]{LLR22} (here $w = 0$ in $(-1,1)\times \overline{B}_R$ by assumption). If $0\notin \omega$, we choose a small open ball $B_R$ such that $\omega \cap B_R =\varnothing$. Then, a similar argument yields that $w=0$ in $(-1,1)\times (\Omega\setminus\overline{B}_r)$ for all $0<r<R$. Thus one can conclude by passing to the limit as $r \to 0^+$. We refer to \cite{ZLZ18} for a singular heat equation, where the authors assumed $0\in \omega$ and the convexity of $\Omega$ for a quantitative unique continuation.
\end{remark}

%%%%%%%%%%%%%%%%%%%%%%%%%%%%%%%%%%%%%%%%%%%%%%%%%%%%%%%%%%%%%%%%%%%%%%%%%
Next, we focus on the non-singular case ($\lambda=0$). We will adapt an idea from \cite[Theorem 1.4]{FF'14} to prove a unique continuation result from sets of positive measure.
\begin{proposition}\label{propuc1}
Let $\Omega \subset \mathbb R^n$ be a bounded domain with a $C^2$ boundary, and let $\omega \subset \Omega$ be a set of positive measure. Consider $w_0$ an $H^1((-1,1)\times \Omega)$ solution to \eqref{eqels} with $\lambda=0$. If $w_0 = 0$ in $(-1,1)\times \omega$, then $w_0 = 0$ in $(-1,1)\times \Omega$.
\end{proposition}
For a general proof, one can refer to \cite[Proposition 3]{DG92}, where the authors used a strong unique continuation and the De Giorgi inequality. Nevertheless, we give a direct proof based on Proposition \ref{lmasy}.

\begin{proof}
Let $w_0$ be a solution to \eqref{eqels} in $(-1,1)\times \Omega$ (with $\lambda=0$) such that $w_0 = 0$ on $E :=(-1,1)\times \omega$. By contradiction we assume that $w_0 \neq 0$ in $(-1,1)\times\Omega$.

For a fixed $\mathbf{x}_0\in E$, the function $v(\mathbf{x}):=w_0(\mathbf{x}_0 +\mathbf{x})$ satisfies \eqref{eqels} (with $\lambda=0$) in $\left((-1,1)\times \Omega\right) - \mathbf{x}_0$. Applying Proposition \ref{lmasy} to $v$, we obtain
\begin{equation}\label{eqasy1}
    r^{-\gamma} v(r (\mathbf{x}-\mathbf{x}_0)) \rightarrow|\mathbf{x}-\mathbf{x}_0|^\gamma \psi\left(\frac{\mathbf{x}-\mathbf{x}_0}{|\mathbf{x}-\mathbf{x}_0|}\right) \quad \text {in } H^1\left(B(\mathbf{x}_0,1)\right) \quad \text { as } r \rightarrow 0^{+},
\end{equation}
where $\psi$ is defined by \eqref{betai}.

Since $|E|>0$, we may suppose that $\mathbf{x}_0$ is a density point of $E$ by Lebesgue's density theorem. Then
$$
\lim_{r \rightarrow 0^{+}} \frac{\left|\left(\mathbb{R}^N \backslash E\right) \cap B(\mathbf{x}_0,r)\right|}{|B(\mathbf{x}_0,r)|}=0.
$$
Then for all $\varepsilon>0$ there exists $r_0=r_0(\varepsilon) \in(0,1)$ such that, for all $r \in\left(0, r_0\right)$,
\begin{equation}\label{ieqm}
    \left|\left(\mathbb{R}^N \backslash E\right) \cap B(\mathbf{x}_0,r)\right| <\varepsilon |B(\mathbf{x}_0,r)|.
\end{equation}
Since $w_0 = 0$ in $E$, by the continuous embedding $H^1(B(\mathbf{x}_0,r)) \hookrightarrow L^{2^*}(B(\mathbf{x}_0,r))$ and Hölder's inequality, \eqref{ieqm} yields
$$
\begin{aligned}
\int_{B(\mathbf{x}_0,r)} w^2_0(\mathbf{x}) d \mathbf{x} & =\int_{\left(\mathbb{R}^N \backslash E\right) \cap B(\mathbf{x}_0,r)} w^2_0(\mathbf{x}) d \mathbf{x} \\
& \leq \left|\left(\mathbb{R}^N \backslash E\right) \cap B(\mathbf{x}_0,r)\right|^{\frac{2^*-2}{2^*}} \left(\int_{\left(\mathbb{R}^N \backslash E\right) \cap B(\mathbf{x}_0,r)}|w_0(\mathbf{x})|^{2^*} d \mathbf{x}\right)^{\frac{2}{2^*}} \\
& \le \varepsilon^{\frac{2^*-2}{2^*}}\left|B(\mathbf{x}_0,r)\right|^{\frac{2^*-2}{2^*}}\left(\int_{\left(\mathbb{R}^N \backslash E\right) \cap B(\mathbf{x}_0,r)}|w_0(\mathbf{x})|^{2^*} d \mathbf{x}\right)^{\frac{2}{2^*}}
\end{aligned}
$$
for all $r \in\left(0, r_0\right)$. By a change of variable, we infer that
$$\begin{aligned}
    &\int_{B(\mathbf{x}_0,1)} r^{-2\gamma}\left|v(r (\mathbf{x}-\mathbf{x}_0))\right|^2 d \mathbf{x}\\
& \le \left(\frac{\omega_{N-1}}{N}\right)^{\frac{2^*-2}{2^*}} \varepsilon^{\frac{2^*-2}{2^*}}\left(\int_{B(\mathbf{x}_0,1)} r^{-2^*\gamma} \left|v(r (\mathbf{x}-\mathbf{x}_0))\right|^{2^*} d \mathbf{x}\right)^{\frac{2}{2^*}}
\end{aligned}$$
for all $r \in\left(0, r_0\right)$. Taking the limit as $r \rightarrow 0^{+}$, we have by \eqref{eqasy1} that
$$
\begin{aligned}
& \int_{B(\mathbf{x}_0,1)}\left|\mathbf{x}-\mathbf{x}_0\right|^{2 \gamma} \psi^2\left(\frac{\mathbf{x}-\mathbf{x}_0}{\left|\mathbf{x}-\mathbf{x}_0\right|}\right) d \mathbf{x}  \\
& \leq\left(\frac{\omega_{N-1}}{N}\right)^{\frac{2^*-2}{2^*}} \varepsilon^{\frac{2^*-2}{2^*}}\left(\int_{B(\mathbf{x}_0,1)}\left|\mathbf{x}-\mathbf{x}_0\right|^{2^*\gamma} \psi^{2^*}\left(\frac{\mathbf{x}-\mathbf{x}_0}{\left|\mathbf{x}-\mathbf{x}_0\right|}\right) d \mathbf{x}\right)^{\frac{2}{2^*}}.
\end{aligned}
$$
Letting $\varepsilon \rightarrow 0^{+}$, we obtain $\psi = 0$ in $\mathbb{S}^{N-1}$. That is, $\displaystyle \sum_{i=j_0}^{j_0+m-1} \beta_i \psi_i = 0$ in $\mathbb{S}^{N-1}$. By the orthonormality of $\left\{\psi_i\right\}$, we get $\beta_i=0$ for all $i\in \{j_0, j_0 +1, \ldots, j_0+m-1\}$. This is in contradiction with \eqref{betai}.
\end{proof}

\begin{remark}
It should be emphasized that the above proof cannot be adapted to the case $\lambda\neq 0$ by a translation argument. Indeed, the proof relies on the fact that the translated function $w_0(\mathbf{x}_0 +\mathbf{x})$ solves the same equation \eqref{eqels} as $w$. This is not the case for $\lambda\neq 0$.
\end{remark}
%%%%%%%%%%%%%%%%%%%%%%%%%%%%%%%%%%%%%%%%%%%%%%%%%%%%%%%%%%
\section{Sharp uniqueness for the singular Schrödinger equation}\label{sec3}
In this section, we first prove a sharp uniqueness result for the singular Schrödinger equation (Theorem \ref{thm1}). Then we apply the result to approximate controllability and uniqueness for an inverse source problem.

\subsection{Uniqueness for the singular Schrödinger equation}
We start by discussing some relevant properties of the Schrödinger operator with an inverse square potential.

We note that in the sequel, $L^2(\Omega)$ (idem for other spaces) will represent the Hilbert space $L^2(\Omega;\mathbb C)$ endowed with the Hermitian inner product given by
$$\langle y,z\rangle_{L^2(\Omega)}:=\int_\Omega y(x) \overline{z(x)}\, dx \qquad \forall y,z\in L^2(\Omega).$$
Following Baras and Goldstein \cite{BG'84}, it is well-known that global existence holds for such an evolution equation when $\lambda \le \lambda_\star$, and instantaneous blow-up occurs when $\lambda >\lambda_\star$. Here we consider the subcritical case $\lambda <\lambda_\star$.

Let us consider the operator defined weakly by
\begin{equation}
    \mathcal{L}_\lambda u :=-\Delta u -\dfrac{\lambda}{|x|^2} u, \qquad D(\mathcal{L}_\lambda):=H^1_0(\Omega).
\end{equation}
Since $\lambda <\lambda_\star$, the Hardy-Poincaré inequality implies that
$$\|y\|_\lambda:=\left[\int_{\Omega}\left(|\nabla y|^2- \frac{\lambda}{|x|^2}y^2\right) d x \right]^{\frac{1}{2}}$$
defines an equivalent norm to the standard norm of $H^1_0(\Omega)$. Thus the operator $\mathcal{L}_\lambda \colon H^1_0(\Omega) \to H^{-1}(\Omega)$ is an isomorphism. By the Rellich-Kondrachov theorem, we see that the part of $\mathcal{L}_\lambda$ on $L^2(\Omega)$ is a self-adjoint operator with a compact inverse. Thus, it generates an analytic $C_0$-semigroup of contractions in $L^2(\Omega)$.

\begin{remark}
The standard singular Schrödinger equation \eqref{(1.2)} is usually stated as
$$\mathrm{i}\partial_t u +\Delta u +\frac{\lambda}{|x|^2} u =0, \qquad (t, x) \in Q,$$
with a positive sign (see, e.g., \cite{VZ'09}), but this is an equivalent equation by complex conjugation.
\end{remark}

\begin{proof}[Proof of Theorem \ref{thm1}]
Let $u$ be a function such that
\begin{equation}\label{Simas}
\begin{cases}
P(t,x,D)u :=\mathrm{i}\partial_t u -\Delta u -\dfrac{\lambda}{|x|^2} u =0, \qquad (t, x) \in Q,\\
u\vert_{(0,T)\times \omega}=0.
\end{cases}
\end{equation}
Let $\mathcal K(t,\tau)\in C^\infty([-1,1]\times[0,T])$ be a solution to the boundary controllability problem for the Schrödinger equation
\begin{equation}
\begin{cases}\mathrm{i}\partial_{\tau}\mathcal K-\partial^2_{t}\mathcal K =0,\qquad   t\in (-1,1),\,\,\tau\in (0,T),\\
\mathcal K(-1,\tau)=\psi (\tau), \quad K(1,\tau)=v(\tau) \;\;\mbox{on}\,\, (0,T),\\
\mathcal K(\cdot,0)=\mathcal K(\cdot, T)=0,
\end{cases} \label{eqcon}
\end{equation}
where $\psi\in C^\infty_0(0,T)$ is a fixed complex valued function, and $v$ is a boundary control. We refer to \cite{MRR} for the existence of a smooth solution to the above controllability problem.

Next, we consider the integral transform defined by
\begin{equation}\label{IT}
    w(t,x)=\int_{0}^T\mathcal K(t,\tau)u(\tau,x)d\tau.
\end{equation}
We claim that the function $w$ satisfies
\begin{equation}\label{simas}
\begin{cases}
L(t,x,D) w :=-\Delta_{t,x} w(t,x) -\dfrac{\lambda}{|x|^2} w(t,x)=0 \qquad \mbox{in}\;  (-1,1)\times \Omega,\\
w=0 \qquad \mbox{in}\; (-1,1)\times \omega.
\end{cases}
\end{equation}
Indeed, if we set $A(x,D') w :=-\Delta w - \dfrac{\lambda}{|x|^2}$, we have for all $p\in H^2_0((-1,1)\times \Omega)$ that
\begin{align*}
&\left\langle w,L^*(t,x,D)p \right\rangle_{L^2((-1,1)\times \Omega)}=\left\langle w,-\partial_{t}^2 p +A^*(x,D')p\right\rangle_{L^2((-1,1)\times \Omega)}\\
&=\left\langle \int_{0}^T \mathcal K(t,\tau)u(\tau,x) d\tau, -\partial_{t}^2p+A^*(x,D')p\right\rangle_{L^2((-1,1)\times \Omega)}\\
&=\left\langle u,-\int_{-1}^1 \overline{\mathcal K}(t,\tau)\partial_{t}^2 p\, dt +A^*(x,D')\int_{-1}^1\overline{\mathcal K}(t,\tau)p\, dt \right\rangle_{L^2(Q)}\\
&=\left\langle u,\mathrm{i}\partial_{\tau}\int_{-1}^1\overline{\mathcal K}(t,\tau)p\, dt + A^*(x,D')\int_{-1}^1\overline{\mathcal K}(t,\tau)p\, dt\right\rangle_{L^2(Q)}\\
&=\left\langle u,P^*(\tau,x,D)\tilde p\right\rangle_{L^2(Q)}=0,
\end{align*}
with
$\displaystyle \tilde p(\tau,x)=\int_{-1}^{1}\overline{\mathcal K}(t,\tau)p(t,x)dt$, where we have used the equations $\eqref{Simas}_1$ and \eqref{eqcon}.  Moreover, $\eqref{Simas}_2$ implies the second equality in \eqref{simas}. This proves the claim \eqref{simas}. By the unique continuation result of Proposition \ref{propuc}, we obtain that $w=0$ in $(-1,1)\times \Omega$. In particular
$$
w(-1,x)=\int_{0}^T\mathcal K(-1,\tau)u(\tau,x)d\tau=\int_{0}^T \psi(\tau)u(\tau,x)d\tau=0\quad \mbox{in}\; \Omega.
$$
By density of $C^\infty_0(0,T)$ in $L^2(0,T)$, we obtain $u = 0$ in $Q$. 
\end{proof}

Similarly, by Proposition \ref{propuc1}, we can prove the following result in the non-singular case $\lambda=0$.
\begin{theorem}\label{thm11}
Let $\omega \subset \Omega$ be a set of positive measure. We assume that the function $u \in C([0,T];L^2(\Omega))$ satisfies \eqref{(1.2)} with $\lambda=0$. Then $u= 0$ in $(0,T) \times \omega$ yields $u=0$ in $(0,T) \times \Omega$.
\end{theorem}

\subsection{Approximate controllability}
Next, we apply the uniqueness result to prove the approximate controllability of the singular Schrödinger equation with an inverse square potential. This will be done without assuming geometrical restrictions on the control region, generalizing the classical case $\lambda=0$.

Let $T > 0$ and let $\omega \Subset \Omega$ be a nonempty open set. Let us consider the following controlled problem
\begin{equation} \label{(1.3)}
\begin{cases}\mathrm{i}\partial_t u -\Delta u -\dfrac{\lambda}{|x|^2} u=\mathds{1}_\omega h(t,x), \qquad & (t, x) \in Q, \\
u(t,x)=0,  \qquad &(t, x) \in \Sigma,\\
u(0, x)=u_0, \qquad & x \in \Omega,\end{cases}
\end{equation}
where $\mathds{1}_\omega$ is the indicator function of the set $\omega$.

\begin{definition}
The equation \eqref{(1.3)} is approximately controllable in time $T$ if, for every initial datum $u_0\in L^2(\Omega)$, for every desired datum $u_d\in L^2(\Omega)$, and for every $\varepsilon>0$, there exists a control $h\in L^2((0,T)\times \omega)$ such that the solution of \eqref{(1.3)} satisfies
$$\|u(T,\cdot)-u_d\|_{L^2(\Omega)} \le \varepsilon.$$
\end{definition}

Let us introduce the so-called adjoint problem associated with \eqref{(1.3)}:
\begin{equation} \label{(1.4)}
\begin{cases}\mathrm{i}\partial_t v -\Delta v -\dfrac{\lambda}{|x|^2} v=0, \qquad & (t, x) \in Q, \\
v(t,x)=0,  \qquad &(t, x) \in \Sigma,\\
v(T, x)=v_T, \qquad & x \in \Omega.\end{cases}
\end{equation}
It is well known, see for instance \cite[Theorem 15.5]{Za'20}, that the approximate controllability of the equation \eqref{(1.3)} is equivalent to the following unique continuation property for the adjoint problem \eqref{(1.4)}:
$$v=0 \quad \text{in } (0,T)\times \omega \qquad \Longrightarrow \qquad v=0 \quad \text{in } Q.$$
Consequently, Theorem \ref{thm1} applied to equation \eqref{(1.4)} yields the following approximate controllability result.
\begin{theorem}\label{thmap}
The singular Schrödinger equation \eqref{(1.3)} is approximately controllable in every time $T > 0$ with controls supported in any nonempty open subset $\omega \Subset \Omega$.
\end{theorem}

\begin{remark}
In the non-singular case $\lambda=0$, the counterpart of Theorem \ref{thmap} follows directly from the classical Holmgren’s uniqueness theorem; see for instance \cite{Jo'82}. However, in the singular case, one needs the uniqueness result of Theorem \ref{thm1} for inverse square potentials.
\end{remark}

\subsection{Uniqueness for an inverse source problem}
In this subsection, we apply Theorem \ref{thm1} to prove the uniqueness of an inverse source problem in the singular Schrödinger equation. The ideas are slight modifications of their counterparts in \cite{IY'22}, so we include the proofs for the reader’s convenience.

Let $u\in C([0,T];L^2(\Omega))$ be a solution to the following singular Schrödinger equation
\begin{equation} \label{(1.1)}
\begin{cases}\mathrm{i}\partial_t u -\Delta u -\dfrac{\lambda}{|x|^2} u=f(x) \rho(t), & (t, x) \in Q, \\
u(0, x)=0, & x \in \Omega,\end{cases}
\end{equation}
where $f\in L^2(\Omega)$ and $\rho\in C^1[0,T]$ is a real function.

\noindent\textbf{Inverse source problem.}  We aim to determine the unknown spatial component $f$ of the source term for a given temporal component $\rho$ from the measurement
$$u\rvert_{(0,T) \times \omega},$$
where $\omega \subset \Omega$ is an arbitrarily chosen nonempty open set.

\begin{theorem}\label{3}
Let $(u,f)$ be a pair satisfying \eqref{(1.1)}. We further assume that $u,\partial_t u\in H^{1,2}(Q)$ and that the function $\rho \in C^1[0,T]$ satisfies
\begin{equation}\label{(1.8)}
\rho(0) \ne 0.                
\end{equation}
If $u = 0$ in $(0,T) \times \omega$, then $f = 0$ in $\Omega$.
\end{theorem}

\begin{proof}
Let $u$ be a function satisfying \eqref{(1.1)} and $u = 0$ in $(0,T) \times \omega$. We consider the equation of $z(t,x)$ given by
\begin{equation}\label{(4.2)}
 (Kz)(t,x)=\partial_t u(t,x), \qquad (t, x) \in Q,
\end{equation}
where the operator $K$ is defined by
\begin{equation}\label{(4.1)}
(Kv)(t):= \rho(0) v(t) + \int^t_0 \rho'(t-s)v(s) ds, \qquad t\in (0,T).
\end{equation}

Since $\rho(0) \ne 0$, the operator $K$ is a Volterra operator of the second 
kind, and we see that $K^{-1}: H^1(0,T) \longrightarrow H^1(0,T)$ exists and 
is bounded. Therefore, $z(\cdot,x) \in H^1(0,T)$ is well defined for each 
$x \in \Omega$, and
\begin{equation}\label{(4.3)}
\partial_t u(t,x) = \rho(0)z(t,x) + \int^t_0 \rho'(t-s)z(s,x) ds, \qquad (t, x) \in Q,                                   
\end{equation}
since $\partial_tu \in H^{1,2}(Q)$.

Since $u(0,\cdot) = 0$ in $\Omega$ and 
$$
\rho(0)z(t,x) + \int^t_0 \rho'(t-s)z(s,x) ds
= \partial_t\left( \int^t_0 \rho(t-s)z(s,x) ds\right), 
$$
we obtain
$$u(t,x) = \int^t_0 \rho(t-s)z(s,x) ds, \qquad (t, x) \in Q,
$$
that is,
\begin{equation}\label{(4.4)}
u(t,x) = \int^t_0 \rho(s)z(t-s,x) ds, \qquad (t, x) \in Q.
\end{equation}

We claim that $z \in H^{1,2}(Q)$ satisfies 
\begin{equation}\label{(4.5)}
\mathrm{i}\partial_t z(t,x) + A(x,D')z(t,x)=0 \qquad \mbox{in } (t, x) \in Q,
\end{equation}
and 
\begin{equation}\label{(4.6)}
z = 0 \quad \mbox{in $(0,T) \times \omega$}.
\end{equation}
We can easily see that \eqref{(4.6)} holds true, since $\partial_tu = 0$ in $(0,T) \times \omega$ implies $Kz = 0$ in $(0,T) \times \omega$. The the injectivity of $K$ implies \eqref{(4.6)}.

Using the embedding $H^{1,2}(Q) \hookrightarrow C([0,T];L^2(\Omega))$ and the equality \eqref{(4.3)}, we deduce
$$
\partial_tu(0,x) = \rho(0)z(0,x), \qquad x\in \Omega.
$$
On the other hand, since $u(0,\cdot)=0$ and $f\rho \in C^1([0,T];L^2(\Omega))$, the mild solution of \eqref{(1.1)} becomes a classical one. Then substituting $t=0$ in \eqref{(1.1)} we obtain
$$
\partial_tu(0,x) = -\mathrm{i} \rho(0)f(x), \qquad x\in \Omega.
$$
Hence $\rho(0)z(0,x) = -\mathrm{i} \rho(0)f(x)$ for $x\in \Omega$.  By $\rho(0) \ne 0$, we 
reach 
\begin{equation}\label{(4.7)}
z(0,x) = - \mathrm{i} f(x), \qquad x\in \Omega.
\end{equation}
Now we will prove (\ref{(4.5)}). By virtue of \eqref{(4.4)} and \eqref{(4.7)}, we have
\begin{align*}
& \partial_tu(t,x) = \rho(t)z(0,x) + \int^t_0 \rho(s)\partial_t z(t-s,x) ds\\
=&-\mathrm{i} \rho(t)f(x) + \int^t_0 \rho(s)\partial_t z(t-s,x) ds
\end{align*}
and
$$
A(x,D')u(t,x) = \int^t_0 \rho(s) A(x,D')z(t-s,x) ds, \qquad (t, x) \in Q.
$$
Consequently, \eqref{(1.1)} implies
\begin{align*}
& - \mathrm{i}\rho(t)f(x) = \partial_tu - \mathrm{i} A(x,D')u(t,x)\\
=& -\mathrm{i} \rho(t)f(x) + \int^t_0 \rho(s)(\partial_t z - \mathrm{i} A(x,D')z)(t-s,x) ds,
\end{align*}
that is,
$$
\int^t_0 \rho(s)(\partial_t z -\mathrm{i} A(x,D')z)(t-s,x) ds = 0, \qquad (t, x) \in Q.
$$
Fix $g\in L^2(\Omega)$ and set $Z(s):= \left\langle (\partial_s z - iA(x,D')z)(s,\cdot), g\right\rangle_{L^2(\Omega)}$ for $s\in (0,T)$. Then we obtain
$$
\int^t_0 \rho(s) Z(t-s) ds = 0, \qquad t\in (0,T).
$$
By Titchmarsh's convolution theorem (see \cite[Theorem VII]{Tit}), 
there exists $t_* \in [0,T]$ such that 
$$
\rho(s) = 0 \quad \mbox{for a.e $s\in (0,T-t_*)$} \quad \text{and} \quad
Z(s) = 0 \quad \mbox{for a.e $s\in (0,t_*)$}.
$$
Since $\rho \in C[0,T]$, we obtain $\rho(s) = 0 \mbox{ for all $s\in (0,T-t_*)$}$. By $\rho(0)\ne  0$, we see that $t^*=T$. Since $g$ is arbitrary, the claim \eqref{(4.5)} is proved.

By using \eqref{(4.5)}-\eqref{(4.6)} and Theorem \ref{thm1}, we obtain $z(0,\cdot)=0$. Then the equality \eqref{(4.7)} implies that $f=0$ in $\Omega$.
\end{proof}

%%%%%%%%%%%%%%%%%%%%%%%%%%%%%%%%%%%%%%%%%%%%%%%%%%%%%%%%%%%%%%%%%%%%%%%%%%%%%%%%
%%%%%%%%%%%%%%%%%%%%%%%%%%%%%%%%%%%%%%%%%%%%%%%%%%%%%%%%%%%%%%%%%%%%%%%%%%%%%%%%

\begin{theorem}\label{4}
Let $u\in C([0,T];L^2(\Omega))$ and $f\in L^2(\Omega)$ satisfy \eqref{(1.1)}. Let $\rho \in C^1[0,T]$ be a function such that
\begin{equation}\label{soplo}
\mbox{$\rho$ is not identically zero on $[0,T]$.}   
\end{equation}
If $u = 0$ in $\Sigma$ and $u = 0$ in $(0,T) \times \omega$, then $f = 0$ in $\Omega$ and $u=0$ in $Q$.
\end{theorem}

\begin{proof}
By scaling time we may restrict to the case
\begin{equation}\label{soplo10}
0\in \mathrm{supp}\, \rho.
\end{equation}
Moreover, we can reduce the problem to the case $\rho(0)=0$ by making the change of variable $\displaystyle w(t,x)=\int_0^{t}u(s,x)ds$.
 
Now, we have 
$$ P(t,x,D)u=f\rho \quad \mbox{in } Q, \qquad u\vert_{\Sigma}=0,\qquad
u=0\quad \mbox{in } (0,T)\times \omega.$$
Next we introduce the function
\begin{equation}\label{popkorn}
y(t,x)=\int_0^{t}\rho(t-s)v(s,x)ds, \qquad x\in \Omega,
\end{equation}
where $v\in C([0,T];L^2(\Omega))$ is a solution to the initial boundary value problem
$$
P(t,x,D)v=0\quad \mbox{in } Q, \qquad v\vert_{\Sigma}=0, \qquad v(0,\cdot)=-\mathrm{i} f.
$$
By definition, the function $y$ satisfies
\begin{equation}\label{lop}
\quad y\vert_{\Sigma}=0\qquad \text{ and }\qquad  y(0,\cdot)=0.
\end{equation}
Moreover, a simple computation shows that
\begin{equation}\label{lopp}
P(t,x,D)y=f\rho\quad \mbox{in } Q.
\end{equation}
Indeed, since $\rho(0)=0$, for any $p\in H^{1,2}(Q), \; p(T,\cdot)=0, \; p\vert_{\Sigma}=0$, we have
\begin{align*}
    \left \langle y,P^*(t,x,D)p \right\rangle_{L^2(Q)}&=\left \langle y, \mathrm{i}\partial_t p +A^*(x,D')p\right\rangle_{L^2(Q)}\\
    & \hspace{-3cm}=\left \langle \int_0^{t}\rho(t-s)v(s,x)ds, \mathrm{i}\partial_t p +A^*(x,D')p\right\rangle_{L^2(Q)}\\
    &\hspace{-3cm}=\left \langle \int_0^{t}\rho(t-s)v(s,x)ds, A^*(x,D')p\right\rangle_{L^2(Q)}-\left \langle \int_0^{t}\partial_t \rho(t-s)v(s,x)ds, \mathrm{i} p \right\rangle_{L^2(Q)}\\
    &\hspace{-3cm}=\left \langle \int_0^{t}\rho(t-s)v(s,x)ds, A^*(x,D')p\right\rangle_{L^2(Q)}+\left \langle \int_0^{t}\partial_s \rho(t-s)v(s,x)ds, \mathrm{i} p \right\rangle_{L^2(Q)}\\
    &\hspace{-3.2cm}=\left \langle \int_0^{t}\rho(t-s)A(x,D')v(s,x)ds, p\right\rangle_{L^2(Q)}+\left \langle \int_0^{t} \rho(t-s) \mathrm{i} \partial_s v(s,x)ds, p \right\rangle_{L^2(Q)}\\
    &\hspace{-2.8cm} + \left \langle f\rho,p\right\rangle_{L^2(Q)}\\
    &\hspace{-3cm}=\left \langle \int_0^{t}\rho(t-s) P(s,x,D)v(s,x), p\right\rangle_{L^2(Q)} + \left \langle f\rho,p\right\rangle_{L^2(Q)}\\
    &\hspace{-3cm}= \left \langle f\rho,p\right\rangle_{L^2(Q)}.
\end{align*}
Hence, the function $y$ is a solution to \eqref{lopp}-\eqref{lop}. By uniqueness of the solution to this initial-boundary value problem, we obtain $y=u$ and then
$$ y\vert_{(0,T)\times \omega}=0.$$
This equality implies 
$$
\int_0^{t}\rho(t-s) v(s,x)ds=0\qquad\mbox{in } (0,T)\times \omega.
$$
Thus
$$
\int_0^{t}\rho(t-s) \langle v(s,\cdot ), g\rangle_{L^2(\omega)}ds=0,
$$
for all $g\in L^2(\omega)$. By Titchmarsh's theorem, there exists $t_* \in [0,T]$ such that $\rho=0$ for a.e $s\in(0,T-t_*)$ and $\langle v(s,\cdot ), g\rangle_{L^2(\omega)}$ for a.e $s\in (0,t_*)$. As in the proof of Theorem \ref{3}, we obtain $t_*=T$ and 
$$
v=0\qquad \mbox{in } (0,T)\times \omega.$$
Theorem \ref{thm1} implies that 
$$
v(t,x)=0 \qquad \mbox{in } Q.
$$
Therefore, by \eqref{popkorn}, we infer that
$$y(t,x)=0 \qquad \mbox{in } Q.$$
Then \eqref{lopp} yields that
$$
\rho(t) f(x)=0\qquad \mbox{in } Q.
$$
Since the function $\rho\not\equiv 0$ in $(0,T)$, we obtain $f= 0$ in $\Omega$.
\end{proof}

\begin{remark}
In the non-singular case $\lambda=0$, Theorem \ref{3} and Theorem \ref{4} hold true assuming only that $\omega \subset \Omega$ is a set of positive measure. It is sufficient to apply Theorem \ref{thm11} instead of Theorem \ref{thm1} in the previous proofs.
\end{remark}

\section{Conclusion}
In this paper, we have investigated the sharp uniqueness of a singular Schrödinger equation with an inverse square potential in the subcritical case $\lambda <\lambda_\star$. Inspired by \cite{IY'22} for bounded coefficients, we have proven similar uniqueness results for singular potentials which are considered critical. The uniqueness result has been applied to approximate controllability and an inverse source problem. This has been done without assuming geometrical conditions on the control/measurement set. It should be noted that imposing some geometrical conditions is quite common in the context of controllability and inverse problems. 

Although we have considered an inverse square potential for simplicity, our results can be adapted for more general cylindrical and multi-body potentials: let $3 \leq k \leq n$ and define the sets $\mathcal{A}_k:=\{J \subseteq\{1,2, \ldots, n\} \colon \mathrm{card}(J)=k\}$ and $
\mathcal{B}_k:=\left\{\left(J_1, J_2\right) \in \mathcal{A}_k \times \mathcal{A}_k \colon J_1 \cap J_2=\varnothing \text { and } J_1<J_2 \,(\text{alphabetic order})\right\}$. For $x=\left(x_1, x_2, \ldots, x_n\right) \in \mathbb{R}^n$ and $J \in \mathcal{A}_k$, we denote by $x_J$ the $k$-tuple $\left(x_i\right)_{i \in J}$ and $\displaystyle \left|x_J\right|^2=\sum_{i \in J} x_i^2$. 

Making use of the above notations, we can show similar results for the general equation
\begin{equation} \label{eqc1}
\begin{cases} \displaystyle
\mathrm{i}\partial_t u -\Delta u -\sum_{J \in \mathcal{A}_k} \frac{\lambda_J}{\left|x_J\right|^2} u -\sum_{\left(J_1, J_2\right) \in \mathcal{B}_k} \frac{\lambda_{J_1 J_2}}{\left|x_{J_1}-x_{J_2}\right|^2} u=f(x) \rho(t), & (t, x) \in Q, \\
u(0, x)=0, & x \in \Omega,
\end{cases}
\end{equation}
where $\lambda_J$ and $\lambda_{J_1 J_2}$ are real constants such that
$$\sum_{J \in \mathcal{A}_k} \max \{\lambda_J, 0\} + \sum_{\left(J_1, J_2\right) \in \mathcal{B}_k} \max\{\lambda_{J_1 J_2}, 0\} <\lambda_\star(k).$$
We refer to \cite{FFT'12} for more details.

In the present paper, we have studied the case of an interior singularity, that is $0\in \Omega$. It would be of much interest to investigate the case when the singularity arises at the boundary, i.e., $0\in \Gamma$.

\section*{Acknowledgment}
I would like to thank the anonymous referees for careful reading and invaluable comments which led to this improved version.


\begin{thebibliography}{99}

\bibitem{BG'84}
\newblock P. Baras and J. Goldstein, 
\newblock The heat equation with a singular potential, 
\newblock {\it Trans. Amer. Math. Soc.}, {\bf 284} (1984), 121--139.

\bibitem{BKP'05}
\newblock L. Baudouin, O. Kavian and J-P. Puel,
\newblock Regularity for a Schrödinger equation with singular potentials and application to bilinear optimal control,
\newblock {\it J. Diff. Equ.}, {\bf 216} (2005), 188--222.

\bibitem{BM'08}
\newblock L. Baudouin and A. Mercado,
\newblock An inverse problem for Schrödinger equations with discontinuous main coefficient,
\newblock {\it Appl Anal.}, {\bf 87} (2008), 1145--1165.

\bibitem{BP'08}
\newblock L. Baudouin and J-P. Puel,
\newblock (Corrigendum) Uniqueness and stability in an inverse problem for the Schrödinger equation,
\newblock {\it Inverse Problems}, {\bf 23} (2008), 1327--1328.

\bibitem{BP'02}
\newblock L. Baudouin and J-P. Puel,
\newblock Uniqueness and stability in an inverse problem for the Schrödinger equation,
\newblock {\it Inverse Problems}, {\bf 18} (2002), 1537--1554.

\bibitem{Ca'12}
\newblock C. Cazacu,
\newblock Schrödinger operators with boundary singularities: Hardy inequality, Pohozaev identity and controllability results,
\newblock {\it J. Funct. Anal.}, {\bf 263} (2012), 3741--3783.

\bibitem{Ch23}
\newblock M. Choulli,
\newblock Quantitative strong unique continuation property for the Schrödinger operator with unbounded potential,
\newblock (2023), arXiv: 2309.08651.

\bibitem{DG92}
\newblock D. De Figueiredo and J-P. Gossez,
\newblock Strict monotonicity of eigenvalues and unique continuation,
\newblock {\it Comm. Partial Differential Equations}, \textbf{17} (1992), 339--346.

\bibitem{FF'14}
\newblock M. M. Fall and V. Felli,
\newblock Unique continuation property and local asymptotics of solutions to fractional elliptic equations,
\newblock {\it Comm. Partial Differential Equations}, {\bf 39} (2014), 354--397.
 
\bibitem{FFT'12}
\newblock V. Felli, A. Ferrero and S. Terracini,
\newblock On the behavior at collisions of solutions to Schrödinger equations with many-particle and cylindrical potentials,
\newblock {\it Discrete Contin. Dyn. Syst.}, {\bf 32} (2012), 3895--3956.

\bibitem{FMS'24}
\newblock G. Fragnelli, A. Moumni and J. Salhi,
\newblock Controllability and stabilization of a degenerate/singular Schrödinger equation,
\newblock {\it J. Math. Anal. Appl.}, {\bf 537} (2024), 128290.

\bibitem{FM'21}
\newblock G. Fragnelli and D. Mugnai,
\newblock {\it Control of Degenerate and Singular Parabolic Equations: Carleman Estimates and Observability},
\newblock Springer Nature, 2021.

\bibitem{HKYM'19}
\newblock X. Huang, Y. Kian, E. Soccorsi and M. Yamamoto,
\newblock Carleman estimate for the Schrödinger equation and application to magnetic inverse problems,
\newblock {\it J. Math. Anal. Appl}, {\bf 474} (2019), 116--142.

\bibitem{IY'022}
\newblock O. Y. Imanuvilov and M. Yamamoto,
\newblock Determination of a source term in Schrödinger equation with data taken at final moment of observation,
\newblock {\it Commun. Anal. Comp.} (2024), doi: 10.3934/cac.2024016.

\bibitem{IY'22}
\newblock O. Y. Imanuvilov and M. Yamamoto,
\newblock Sharp uniqueness and stability of solution for an inverse source problem for the Schrödinger equation,
\newblock {\it Inverse Problems} (2023), 105013 (19pp).

\bibitem{JK'85}
\newblock D. Jerison and C. E Kenig,
\newblock Unique continuation and absence of positive eigenvalues for Schrödinger operators,
\newblock {\it Ann. of Math.}, {\bf 121} (1985), 463--494.

\bibitem{Jo'82}
\newblock F. John,
\newblock {\it Partial Differential Equations},
\newblock (4. ed), Springer, 1982.

\bibitem{LLR22}
\newblock J. Le Rousseau, G. Lebeau and L. Robbiano,
\newblock {\it Elliptic Carleman Estimates and Applications to Stabilization and Controllability. Volume I. Dirichlet Boundary Conditions on Euclidean Space},
\newblock Prog. Nonlinear Differ. Equ. Appl., vol. 97, Cham: Birkhäuser (2022).

\bibitem{Le'92}
\newblock G. Lebeau,
\newblock Contrôle de l'équation de Schrödinger,
\newblock {\it J. Math. Pures Appl.}, {\bf 71} (1992), 267--291.

\bibitem{Ma'94}
\newblock E. Machtyngier,
\newblock Exact controllability for the Schrödinger equation,
\newblock {\it SIAM J. Control and Optim.}, {\bf 32} (1994), 24--34.

\bibitem{MRR}
\newblock P. Martin, L. Rosier and P. Rouchon,
\newblock Controllability of the 1D Schrödinger equation using flatness,
\newblock {\it Automatica}, {\bf 91} (2018), 208--216.

\bibitem{MM'23}
\newblock A. Mercado and R. Morales,
\newblock Exact controllability for a Schrödinger equation with dynamic boundary conditions,
\newblock {\it SIAM J. Control Optim.}, {\bf 61} (2023), 3501--3525.

\bibitem{MOR'08}
\newblock A. Mercado, A. Osses and L. Rosier,
\newblock Inverse problems for the Schrödinger equation via Carleman inequalities with degenerate weights,
\newblock {\it Inverse Problems}, {\bf 24} (2008), 015017.

\bibitem{Pa'92}
\newblock Y. F. Pan,
\newblock Unique continuation for Schrödinger operators with singular potentials,
\newblock {\it Comm. Partial Differential Equations}, {\bf 17} (1992), 953--965.

\bibitem{SV23}
\newblock A. Shao and B. Vergara,
\newblock Approximate boundary controllability for parabolic equations with inverse square infinite potential wells,
\newblock {\it Nonlinear Analysis}, {\bf 248}, (2024), 113624.

\bibitem{Tit}
\newblock E. C. Titchmarsh,
\newblock The zeros of certain integral functions,
\newblock {\it Proc. London Math. Soc.}, {\bf 25} (1926), 283--302.

\bibitem{VZ'09}
\newblock J. Vancostenoble and E. Zuazua,
\newblock Hardy inequalities, observability, and control for the wave and Schrödinger equations with singular potentials,
\newblock {\it SIAM J. Math. Anal.}, {\bf 41} (2009), 1508--1532.

\bibitem{XWD23}
\newblock H. Xu, L. Wei and Z. Duan,
\newblock Observability and unique continuation inequalities for the Schrödinger equations with inverse-square potentials, \newblock (2023), arXiv: 2310.20541.

\bibitem{YM}
\newblock G. Yuan and M. Yamamoto,
\newblock Carleman estimates for the Schrödinger equation and applications to an inverse problem and an observability inequality,
\newblock {\it Chin. Ann. Math.}, {\bf 31} (2010), 555--578.

\bibitem{Za'20}
\newblock J. Zabczyk,
\newblock {\it Mathematical Control Theory: An Introduction},
\newblock Systems \& Control: Foundations \& Applications, Springer Nature Switzerland, 2nd ed., 2020.

\bibitem{ZLZ18}
\newblock G. Zheng, K. Li and Y. Zhang,
\newblock Quantitative unique continuation for the heat equations with inverse square potential,
\newblock {\it J. Inequal. Appl.}, {\bf 310} (2018), 1--13.

\bibitem{Zu'03}
\newblock E. Zuazua,
\newblock Remarks on the controllability of the Schrödinger equation,
\newblock {\it CRM Proc., Lecture Notes}, {\bf 33}, Amer. Math. Soc., Providence, RI, 2003, 193--211.
    
\end{thebibliography}
\end{document}